\documentclass[11pt,a4paper]{amsart}
\usepackage[latin1]{inputenc}
\usepackage{amssymb}
\usepackage{amsmath}
\usepackage{bm}
\usepackage{ifthen}
\usepackage{color}
\usepackage{dsfont}
\usepackage[all,cmtip]{xy}
\usepackage{ulem}

\newtheorem{theorem}{Theorem}[section]
\newtheorem{rem}[theorem]{Remark}
\newtheorem{prop}[theorem]{Proposition}
\newtheorem{lemma}[theorem]{Lemma}
\newtheorem{cor}[theorem]{Corollary}

\newtheorem{example}[theorem]{Example}

\def\hess{\operatorname{hess}}

\newcommand{\Sing}{\operatorname{Sing}}

\renewcommand{\O}{{\mathcal O}}

\newcommand{\p}{{\mathbb P}}

\newcommand{\Hm}{\operatorname{H}}

\newcommand{\Bl}{\operatorname{Bl}}

\newcommand{\codim}{\operatorname{codim}}

\newcommand{\Pf}{\operatorname{Pf}}

\newcommand{\Ann}{\operatorname{Ann}}

\newcommand{\rk}{\operatorname{rk}}

\renewcommand{\P}{\mathbb P}
\renewcommand{\vert}{\operatorname{Vert}}
\newcommand{\Jac}{\operatorname{Jac}}

\newcommand{\map}{\dasharrow}

\renewcommand{\vert}{\operatorname{Vert}}

\renewcommand{\th}{\operatorname{th}}

\newcommand{\reg}{\operatorname{reg}}

\newcommand{\Hom}{\operatorname{Hom}}

\renewcommand{\th}{\operatorname{th}}

\newcommand{\Res}{\operatorname{Res}}

\newcommand{\parder}[3][Default]{
	\frac{\partial \ifthenelse{\equal{#1}{Default}}{}{^{#1}}#2}{
              \partial #3 \ifthenelse{\equal{#1}{Default}}{}{^{#1}}}}

\begin{document}


\title[Hypersurfaces with vanishing hessian via Dual Cayley Trick]{Hypersurfaces with vanishing hessian via
Dual Cayley Trick}

\author[R. Gondim]{Rodrigo Gondim$\dagger$}
\address{Universidade Federal Rural de Pernambuco}
\email{rodrigo.gondim.neves@gmail.com}
\author[F. Russo]{Francesco Russo*}
\address{Dipartimento di Matematica e Informatica, Universit\` a degli Studi di Catania, Viale A. Doria 5, 95125 Catania, Italy}
\email{frusso@dmi.unict.it, giovannistagliano@gmail.com}
\author[G. Staglian\` o]{Giovanni Staglian\` o}
\thanks{$\dagger$Partially supported by the CAPES postdoctoral fellowship, Proc. BEX 2036/14-2 and by the FIR2014  {\it Aspetti geometrici e algebrici della Weak e Strong Lefschetz Property} of the University of Catania}
\thanks{*Partially  supported  by the PRIN {\it Geometria delle variet\`{a} algebriche} and by the FIR2014  {\it Aspetti geometrici e algebrici della Weak e Strong Lefschetz Property} of the University of Catania; the author is a member of the G.N.S.A.G.A. of INDAM}

\begin{abstract}
We present
a general construction of hypersurfaces with vanishing hessian, starting from any irreducible non-degenerate
variety  whose dual variety is a hypersurface and based on the so called  Dual Cayley Trick.
The geometrical  properties of these hypersurfaces 
are different from the series known until now. In particular, their  dual varieties can have arbitrary codimension in the image
of the associated polar map.

\end{abstract}

\maketitle

\section*{Introduction}

If  $X = V(f) \subset \P^N$ is  a  hypersurface,  
the hessian determinant of $f$ (from now on simply  called   the hessian of $f$ or, by abusing language, the hessian  of $X$)  is the determinant of the 
hessian matrix of $f$.

Hypersurfaces with vanishing hessian 
were studied systematically  for the first time in  the fundamental paper \cite{GN}, where  P. Gordan and M. Noether  analysed
O. Hesse's claims in \cite{Hesse1, Hesse2}
according to which these hypersurfaces should be necessarily cones.  Clearly the claim is true if $\deg(f)=2$ 
so that the first relevant case for the problem is that of cubic hypersurfaces. The cubic hypersurface
$V(x_0x_3^2 +  x_1x_3x_4 +  x_2x_4^2)\subset \p^4$ has vanishing hessian but
it is not a cone, see \cite{Perazzo}. By adding the term $\sum_{i=5}^Nx_i^3$ we get examples of
 irreducible cubic hypersurfaces in $\p^N$ with $N\geq 4$ with vanishing hessian 
 which are not cones.

Notwithstanding, the  question is quite subtle
because,  as it was firstly pointed out in \cite{GN},
Hesse's  claim is true for $N\leq 3$,  see also \cite{Lossen, GR} and \cite[Section 7]{Russo}. 
Moreover,  hypersurfaces with vanishing hessian  and the Gordan-Noether Theory developed in \cite{GN} have a  wide range of applications
in different areas of mathematics such as Algebraic and  Differential Geometry (see \cite{AG, AG2, FP}), Commutative Algebra and the theory of EDP (see \cite{DeBondt1, Russo, Rodrigo}), Approximation Theory and Theoretical Physics (see \cite{FP,AG}) and 
Combinatorics (see \cite{GoZa}). 
 
Hypersurfaces with vanishing hessian have been forgotten by  algebraic geometers  for a long time and recently they were rediscovered in other contexts.
For example the cubic hypersurface in $\p^4$ recalled above is  celebrated nowadays in the modern differential geometry literature as the 
{\it Bourgain-Sacksteder Hypersurface} (see \cite{AG2, AG, FP}).  

Many classes of hypersurfaces with vanishing hessian, which are not cones, are ruled by a family of linear
spaces
along which the hypersurface is not developable. In particular, this ruling is different from the one  given by the fibers of the Gauss map. These examples and their generalizations are known in differential geometry as
{\it twisted planes}, see for example \cite{FP}. Despite the huge number of papers dedicated to this subject by differential geometers very few classification or structure results have been obtained.
Moreover, the global point of view provided by polarity, used systematically in this paper, has been completely overlooked in other areas.

The known  series of examples of hypersurfaces $X\subset \p^N$ with vanishing hessian, which are not cones, 
have been  constructed by  Gordan and Noether, Perazzo, Franchetta, Permutti, see \cite{GN, Perazzo, Franchetta, Permutti1, Permutti2},
and later have been revisited and generalised in \cite[Section 2]{CRS}, see also \cite{CuRaSi, CuRaSi2}.

All these examples share several geometrical behaviours, see  in particular Subsection \ref{geopro}. For instance 
there exists  a linear subspace $L \subset \Sing X$, dubbed the {\it core of $X$} in \cite{CRS}, such that, 
letting  $L_{\alpha} = \P^{k+1}\supset L$ and letting 
$L_{\alpha} \cap X = \mu_{\alpha} L \cup X_{\alpha}$ for some $\mu_{\alpha}\in\mathbb N$,  
the variety  $X_{\alpha}$ is a cone with vertex $V_{\alpha} = \operatorname{Vert}(X_{\alpha}) \subset L$ tangent to $Z_X^* \subset L$, where $Z_X\subsetneq (\p^N)^*$ is the closure
of the image of the polar map of $X$. When  the cones $X_{\alpha}$ split into a union of linear spaces, the hypersurface is a twisted plane.
Furthermore,  the dual variety $X^*\subset(\p^N)^*$ of most of the examples in these series 
 tends to be a divisor in  $Z_X$. From the perspective of   Segre's Formula, recalled in Section \ref{prelHess},
 this  means 
that the rank of the hessian matrix of a homogeneous polynomial with vanishing hessian determinant should be equal to the rank of the hessian matrix modulo the ideal generated by $f$, see Section
\ref{prelHess} for precise definitions. This  seemed to be the most natural and general behaviour, at least at a first glance.

On the other hand,  if  
$Y\subset(\p^N)^*$ is an arbitrary  non-degenerate irreducible  variety of  dimension $n\geq 1$, then, after identifying $\p^N$ with $(\p^N)^{**}$,
the dual variety  $X=Y^*\subset\p^N$  is not a cone and, in general,  one expects that  $X$ is a hypersurface with non  vanishing hessian. If this is the case,  $Z_X=(\p^N)^*$ and $\codim(X^*, Z_X)=\codim(Y,(\p^N)^*)=N-n$
is arbitrary large.

These remarks motivate the search of hypersurfaces with vanishing hessian $X\subset\p^N$ such that $X^*$
has arbitrary codimension in $Z_X\subsetneq (\p^N)^*$.
Here we shall  present
a general construction of such hypersurfaces, starting from any irreducible non-degenerate
variety whose dual variety is a hypersurface and  based on the  Dual Cayley Trick.
The geometrical  properties of these hypersurfaces 
are different from those described above (for example $V_\alpha$ is not contained in $L$) and their  dual varieties can have arbitrary codimension in the image
of the polar map. The ubiquity of the examples suggests that the classification of hypersurfaces with vanishing hessian for $N\geq 5$
might be very intricate, perhaps requiring a completely different approach not  based (only) on Gordan--Noether
Theory, which worked for $N\leq 4$, see \cite{GN, Franchetta, GR, Russo}. Other interesting series of examples of hypersurfaces with vanishing hessian such that $\codim(X^*,Z_X)$ is large
 has been recently constructed in \cite{CuRaSi, CuRaSi2} (see also Remark \ref{CuRS} for a possible geometrical description
of this series of examples).

The paper is organised as follow. In Section 1 we fix the notation and introduce the main definitions. Section 2 is devoted to the construction of the first
series of examples, leading to Theorem \ref{codimension} and ending with the description of the geometrical properties of the examples. In Section 3 we
briefly recall the definitions of resultant and discriminant and we apply them to calculate explicitly the dual of rational normal scroll surfaces in Theorem \ref{dualSab}.
Then we introduce the Cayley Trick and the Dual Cayley Trick and apply the Dual Cayley Trick to a generalisation of the series of examples constructed
in Section 2 (see Theorem \ref{codimensionk}), also providing a new conceptual proof of  the part (ii) of Theorem \ref{codimension}. In Section 4 we  prove that the duals of internal projections
from a point of Scorza Varieties have vanishing hessian and we describe the geometrical properties of these hypersuperfaces and of their polar maps.
\medskip

{\bf Acknowledgements}. We  wish to thank the referee for a very careful reading, for pointing out some inaccuracies  and for many useful suggestions leading to a significant improvement of the exposition.

\section[Preliminaries and definitions]{Preliminaries and definitions}\label{prelHess}

 Let  $f(x_0,\ldots,x_N)\in\mathbb  K[x_0,\ldots, x_N]_d$ be a homogeneous polynomial of degree $d\geq 1$  without multiple irreducible factors   and let $X = V(f) \subset \P^N$  be the associated projective  hypersurface. We shall always assume that   $\mathbb K$ is an algebraically closed field
of characteristic zero. Let
$$ \Hm(f)=\left[\frac{\partial^2 f}{\partial x_i \partial x_j}\right]_{0\leq i,j\leq N}$$
be the {\it hessian matrix of $f$}  (or of $X$).

Clearly $\Hm(f)=\mathbf 0_{(N+1)\times(N+1)}$ if and only if $d=1$. Thus, from now on, 
{we shall suppose $d\geq 2$. Let
 $$\hess_f=\det(\Hm(f))$$ 
be the  {\it hessian (determinant) of $f$} (or of $X$, in which case 
it will be denoted by  $\hess_X$, which 
is defined modulo multiplication by a non zero element in~$\mathbb K$). 
\medskip

There are two possibilities:

 \begin{enumerate}
\item either $\hess_f= 0$
\medskip
or
\item $\hess_f\in\mathbb K[x_0,\ldots, x_N]_{(N+1)(d-2)}$.
\end{enumerate}
\medskip

We shall be interested in case (1), that is in {\it hypersurfaces with vanishing hessian} ({\it determinant}).
\medskip

\subsection{The polar map}\label{ss:polarmap}
Let
$$
\nabla_f=\nabla_X:\p^N\dasharrow(\p^N)^{*}$$
be {\it the polar (or gradient) map of $X=V(f)\subset\p^N$} which associates to $p\in \p^N$
the {\it polar hyperplane to $X$ with respect to $p$.} In coordinates it is defined by
$$\nabla_f(p)=\Big(\frac{\partial f}{\partial x_0}(p):\cdots :\frac{\partial f}{\partial x_N}(p)\Big).$$
\index{Polar map $\nabla_f$ o $\nabla_X$}\index{$\nabla_f$}\index{$\nabla_X$}

The base locus scheme of $\nabla_f$ is the scheme

$$\operatorname{Sing}(X):=V\Big(\frac{\partial f}{\partial x_0},\ldots, \frac{\partial f}{\partial x_N}\Big)\subset \p^N.$$

Let
   $$Z_X := \overline{\nabla_f(\P^N)}\subseteq(\p^N)^{*}$$
   be  {\it the  polar image of $\p^N$}.\index{$Z$}\index{Polar image}

We can consider the rational map $\nabla_f$
as the quotient by the natural $\mathbb K^*$-action of the affine morphism
$$\nabla_f:\mathbb K^{N+1}\to\mathbb K^{N+1}$$
defined in the same way. Thus we have
the following key formula:
\begin{equation}\label{KeyFormula}
\Hm(f)=\Jac(\nabla_f:\mathbb K^{N+1}\to\mathbb K^{N+1}),
\end{equation}
that is, the hessian matrix of $f$ is the Jacobian matrix of the affine morphism $\nabla_f$.
Hence, $\hess_f=0$ if and only if  $Z_X\subsetneq \p^N$ so that $\hess_f=0$ if and only if $\frac{\partial f}{\partial x_0}, \frac{\partial f}{\partial x_1}, \ldots, \frac{\partial f}{\partial x_N}$
are algebraically dependent.

The restriction of $\nabla_f$ to $X$ is the Gauss map of $X$:
$$
\begin{array}{cccc}
\mathcal G_X=\nabla_{f|X}:& X&\dasharrow&\mathbb (\p^N)^{*}\\
&X_{\reg}\ni p  &\longmapsto& \mathcal G_X(p)=[T_pX]
\end{array}
$$
which to a non-singular point $p\in X_{\reg}$ associates the point $[T_pX]\in(\p^N)^*$
representing the projective tangent hyperplane $T_pX$ to $X$ at the smooth point $p$.
Then, by definition, 
$$X^*:=\overline{\mathcal G_X(X)}\subseteq Z_X$$
is the {\it dual variety of $X$}.
\medskip

If $A$ is a matrix with entries in $\mathbb K[x_0,\ldots, x_N]$ and if $f\in \mathbb K[x_0,\ldots,x_N]$,
then $\rk_{(f)}A$ denotes the {\it rank of $A$ modulo $(f)$}, that is the maximal order
of a minor not belonging to the ideal generated by $f$. With this notation, obviously $\rk A=\rk_{(0)}A$.
\medskip

\begin{lemma}\label{rkGauss}{\rm (\cite{Segre})}
Let $X=V(f)=X_1\cup\cdots\cup X_r\subset\p^N$ be a reduced hypersurface with $X_i=V(f_i)$,  $f=f_1\cdots f_r$
and $f_i$ irreducible. Then:
\medskip
\begin{enumerate}
\item[(i)] If $p_i\in X_i$ is general, then
\begin{equation}\label{Segreformula}
\rk(\operatorname{d}\mathcal G_X)_{p_i}=\rk_{(f_i)}H(f)-2.
\end{equation}
In particular, 
\begin{equation}\label{Segreformulaeq}
\dim(X_i^*)=\rk(\operatorname{d}\mathcal G_X)_{p_i}=\rk_{(f_i)}H(f)-2\leq \rk H(f)-2= \dim(Z_X)-1.
\end{equation}
\medskip
\item[(ii)] If $X$ is irreducible, then $f^{N-\dim(X^*)-1}$ divides $\hess_f$.
\end{enumerate}
\end{lemma}
\medskip

We point out an immediate consequence for further reference.
\medskip

\begin{cor}\label{XdualinZ} Let $X=V(f)\subset\p^N$ be a reduced hypersurface with vanishing
hessian. Then
$$X^*\subsetneq Z_X\subsetneq (\p^N)^*.$$
\end{cor}
\medskip

We shall also need the following remark.
\medskip

\begin{lemma}\label{projectionHess}{\rm (\cite[Lemma
3.10]{CRS})}
Let  $X=V(f)\subset\p^N$ be a hypersurface. Let
$H=\p^{N-1}$ be a hyperplane not contained in $X$, let
$h=H^*$ be the corresponding point in $(\p^N)^{*}$ and let
$\pi_h$ denote the projection from the point $h$. Then we have a commutative diagram:
\begin{equation*}
\xymatrix{ H \ar@{-->}[rr]^{\nabla_{X\cap H}} \ar@{-->}[d]_{(\nabla_{X})|_H} & & H^{\ast}\\ (\p^N)^{\ast}\ar@{-->}[rru]_{\pi_h} } 
\end{equation*}
In particular, $\overline{\nabla_{X\cap H}(H)}\subseteq \pi_h(Z_X).$
\end{lemma}
\medskip

\section{Hypersurfaces with vanishing hessian constructed from any non-degenerate variety}

\subsection{Hypersurfaces with vanishing hessian constructed from duals of arbitrary non-degenerate subvarieties}

Let us consider $\p^{2n+1}$ with homogeneous coordinates
$$(u:v:x_1:\dots:x_n:y_1:\cdots:y_n),$$
$\p^1$ with homogeneous coordinates $(s:t)$ and $\p^{n-1}_{\mathbf z}$ with homogeneous
coordinates $(z_1:\cdots:z_n)$, where $\mathbf z=(z_1,\ldots,z_n)$.
Let  $\mathbf x=(x_1,\ldots, x_n)$, let $\mathbf y=(y_1,\ldots, y_n)$,
 let $$\phi_1:\p^{2n+1}\map\p^1$$ be the rational map defined by
$$\phi_1(u:v:\mathbf x:\mathbf y)=(u:v)$$
and let $\phi_2:\p^{2n+1}\map\p^{n-1}_{\mathbf z}$ be the rational map defined by
\begin{equation}\label{phi2}
\phi_2(u:v:\mathbf x:\mathbf y)=(ux_1-vy_1:\cdots:ux_n-vy_n).
\end{equation}

Let $g(z_1,\ldots, z_n)\in\mathbb K[z_1,\ldots,z_n]_d$ be a reduced irreducible polynomial such that 
 the associated  irreducible hypersurface of degree $d$
 $$Y^*=V(g)\subset\p^{n-1}_{\mathbf z}$$
 is  not a cone. This is equivalent 
 to
$$Y=Y^{**}=\overline{\nabla_g(Y^*)}\subset(\p^{n-1}_{\mathbf z})^* $$ 
being non-degenerate.

Let
$$f(u,v,\mathbf x,\mathbf y)=g(ux_1-vy_1,\ldots, ux_n-vy_n)\in \mathbb K[u,v,\mathbf x,\mathbf y]_{2d}$$
and let $X=V(f)\subset\p^{2n+1}$. Clearly,
\begin{equation}\label{XfromY*}
V(f)=\overline{\phi_2^{-1}(V(g))}.
\end{equation}

The partial
derivatives of $f$ are  linearly independent over $\mathbb K$, due to the hypothesis on $g$, so that  $X=V(f)\subset\p^{2n+1}$ is not a cone. 
One verifies that 
$$\Sing(X)=V(u,v)\cup (\p^1\times\p^n)\cup \phi_2^{-1}(\Sing(V(g))),$$
where $\p^1\times\p^n\subset\p^{2n+1}$ is the Segre variety defined by the equations 
$$
\rk
\left( \begin{array}{cccc}
v&x_1&\ldots&x_n\\
u&y_1&\ldots&y_n
\end{array}
\right)
=1.$$

From 
\begin{equation}\label{dxi}
\frac{\partial f}{\partial x_i}=u\frac{\partial g}{\partial z_i}(u\mathbf x-v\mathbf y)
\end{equation}
and 
\begin{equation}\label{dyj}
\frac{\partial f}{\partial y_j}=-v\frac{\partial g}{\partial z_j}(u\mathbf x-v\mathbf y),
\end{equation}
we deduce that, for every $i\neq j$, 
\begin{equation}\label{ij-ji}
\frac{\partial f}{\partial x_i}\frac{\partial f}{\partial y_j}-\frac{\partial f}{\partial x_j}\frac{\partial f}{\partial y_i}=0.
\end{equation}

Thus $X\subset\p^{2n+1}$ has vanishing hessian since the partial derivates of $f$ are algebraically dependent.

\subsection{Polar image and dual variety of $X=V(g(u\mathbf x-v\mathbf y))\subset\p^{2n+1}$}

We need to introduce some more notation.
Let $$(u':v':x'_1:\dots:x'_n:y'_1:\cdots:y'_n)$$ be homogenous coordinates on $(\p^{2n+1})^*,$ dual to the coordinates chosen on $\p^{2n+1}$. 
Let $$L'=V(x_1',\ldots,x'_n,y'_1,\ldots, y'_n)=\p^1_{u',v'}\subset (\p^{2n+1})^*$$ and let $$W=\p^1\times\p^{n-1}\subset V(u',v')=\p^{2n-1}_{\mathbf x',\mathbf y'}\subset(\p^{2n+1})^*$$
be the Segre variety defined by the equations:
\begin{equation}\label{eqW}
\rk
\left( \begin{array}{ccc}
x_1'&\ldots&x'_n\\
y'_1&\ldots&y'_n
\end{array}
\right)
=1.
\end{equation}

Let  $S(L',W)\subset(\p^{2n+1})^*$ be  the  cone with vertex $L'$ over the Segre variety $W=\p^1\times\p^{n-1}$.
Thus $\dim(S(L',W))=n+2$ and $S(L',W)\subset(\p^{2n+1})^*$ is defined by  the equations \eqref{eqW} of $W$. 

Let $\p^{n-1}_{\mathbf z'}$ with homogeneous coordinates $(z'_1:\cdots:z'_n)$ be the dual of $\p^{n-1}_{\mathbf z}$ and consider $\p^n_{\mathbf z'}$ with homogeneous coordinates $(z'_0:z'_1:\cdots:z'_n)$. With this notation   $\p^{n-1}_{\mathbf z'}\subset\p^n_{\mathbf z'}$ is  the hyperplane of equation $z'_0=0$. Given $Y\subset\p^{n-1}_{\mathbf z'}$, let $\widetilde Y\subset\p^n_{\mathbf z'}$ be the cone
over $Y$ with vertex $(1:0:\cdots:0)$.

\begin{theorem}\label{codimension} Let the hypothesis and the notation be as above,  let $$Z_{Y^*}=\overline{\nabla_g(\p^{n-1}_{\mathbf z})}\subseteq\p^{n-1}_{\mathbf z'},$$ 
let
$$\p^1\times Z_{Y^*}\subset(\p^{2n-1})^*$$ be the Segre embedding and let
$X=V(g(u\mathbf x-v\mathbf y))\subset\p^{2n+1}$. Then:
\medskip
\begin{enumerate}
\item[(i)] $Z_X=\overline{\nabla_f(\p^{2n+1})}=S(L',\p^1\times Z_{Y^*})\subset(\p^{2n+1})^*$;
\medskip
\item[(ii)] $X^*=\p^1\times\widetilde Y\subset(\p^{2n+1})^*$ Segre embedded;
\medskip
\item[(iii)] $\codim(X^*,Z_X)=\codim(Y, Z_{Y^*})+1.$
\end{enumerate}
\medskip

In particular, if $g(\mathbf z)$ has non-vanishing hessian determinant, then $Z_{Y^*}=\p^{n-1}_{\mathbf z'}$
and $\codim(X^*,Z_X)=\codim(Y, \p^{n-1})+1.$
\end{theorem}
\begin{proof}
By definition  $\nabla_f:\p^{2n+1}\map(\p^{2n+1})^*$ is given by
$$(\frac{\partial f}{\partial u}:\frac{\partial f}{\partial v}:\frac{\partial f}{\partial x_1}:\cdots:\frac{\partial f}{\partial x_n}:\frac{\partial f}{\partial y_1}:\cdots:\frac{\partial f}{\partial y_n}).$$

From \eqref{ij-ji} and from \eqref{eqW} we deduce that 
\begin{equation}\label{in1}
\overline{\nabla_f(\p^{2n+1})}\subseteq S(L',\p^1\times Z_{Y^*}).
\end{equation}

We also have
\begin{equation}\label{du}
\frac{\partial f}{\partial u}=\sum_{i=1}^nx_i\frac{\partial g}{\partial z_i}(u\mathbf x-v\mathbf y),
\end{equation}
\begin{equation}\label{dv}
\frac{\partial f}{\partial v}=-\sum_{j=1}^ny_j\frac{\partial g}{\partial z_j}(u\mathbf x-v\mathbf y).
\end{equation}

Let
$$p=(\tilde{u}':\tilde{v}': \tilde{\mathbf z}':\lambda\tilde{\mathbf z}')\in S(L',\p^1\times Z_{Y^*})$$
be a general point. In particular, we can suppose  $\tilde u'\neq\lambda \tilde v'$ and that  $\tilde{\mathbf z}'\neq\mathbf 0$. Then $[\tilde{\mathbf z}']\in Z_{Y^*}$ is general and by definition
there exists $\tilde{\mathbf z}\in\p^{n-1}_{\mathbf z}$ such that $\nabla_g(\tilde{\mathbf z})=\tilde{\mathbf z}'$.
Looking at \eqref{dxi} and \eqref{dyj}, we impose $v=-\lambda u$. If $u(\mathbf x+\lambda\mathbf y)=\tilde{\mathbf z}$ and if $v=-\lambda u$ hold,
then $\frac{\partial f}{\partial x_i}=u\tilde{\mathbf z}_i'$ and $\frac{\partial f}{\partial y_i}=u\lambda\tilde {\mathbf  z}_i'$. Hence, to find $q$ such that $\nabla_f(q)=p$, it is sufficient
to determine a solution of the system of equations
\begin{equation}\label{sistema1}
\left\{
\begin{array}{rl}
u(\mathbf x+\lambda\mathbf y)=&\tilde{\mathbf z}\\
\medskip

\sum_{i=1}^nx_i\frac{\partial g}{\partial z_i}(\tilde{\mathbf z})=&u\tilde{u}'\\
\medskip

-\sum_{i=1}^ny_i\frac{\partial g}{\partial z_i}(\tilde{\mathbf z})=&u\tilde{v}'\\
\end{array}
\right.
\end{equation}

From
$$u\tilde{u}'=\sum_{i=1}^n(\frac{\tilde{z_i}}{u}-\lambda y_i)\frac{\partial g}{\partial z_i}(\tilde{\mathbf z})=\frac{1}{u}d\cdot g(\tilde{\mathbf z})+\lambda u\tilde{v}',$$
we deduce 
$$u^2=\frac{d\cdot g(\tilde{\mathbf z})}{\tilde{u}'-\lambda\tilde{v}'}.$$

If $\tilde u$ is a solution of this last equation and if
$\tilde{\mathbf a}=(\tilde{a}_1,\ldots, \tilde{a}_n)$ is a solution of the last linear equation in \eqref{sistema1},  then
$$p=\nabla_f(\tilde u:-\lambda \tilde u:\frac{\tilde{\mathbf z}}{\tilde u}-\lambda\tilde{\mathbf a}:\tilde{\mathbf a}),$$
yielding equality in \eqref{in1}.

\medskip

By restricting $\nabla_f$ to $X$, we deduce  that 
$$X^*=\overline{\nabla_f(X)}\subseteq S(L', \p^1\times Z_{Y^*}).$$

Let $T=\p^1\times \p^n\subset  (\p^{2n+1})^*$ be the Segre variety defined by the equations
$$
\rk
\left( \begin{array}{cccc}
v' & x'_1 & \ldots& x'_{n'}\\
u'&y'_1&\ldots&y'_n
\end{array}
\right)
=1.$$
Since $\deg(f)=2d$, Euler's Formula gives
\begin{equation}\label{Eulerf}
(2d)f=u\frac{\partial f}{\partial u}+v\frac{\partial f}{\partial v}+\sum_{i=1}^nx_i\frac{\partial f}{\partial x_i}+\sum_{i=1}^ny_i \frac{\partial f}{\partial y_i}=2(u\frac{\partial f}{\partial u}+v\frac{\partial f}{\partial v}).
\end{equation}

Thus,  for every $p\in X$, we have 
\begin{equation}\label{extraeq}
u(p)\frac{\partial f}{\partial u}(p)+v(p)\frac{\partial f}{\partial v}(p)=0, 
\end{equation}
yielding
$$\overline{\nabla_f(X)}=X^*\subseteq T\subset (\p^{2n+1})^*.$$

Indeed, for $i\neq j$ the equations $x'_iy'_j-x_jy'_i=0$ are satisfied by any point $\nabla_f(p)$ due to \eqref{ij-ji}. Due to \eqref{extraeq}, the equations
$v'y'_i-u'x'_i=0$ are satisfied by $\nabla_f(p)$ for every $p\in X$. 
Furthermore, for every $(\mu:\nu) \in\p^1_{\mathbb K}$, the hypersurface $X\cap V(\mu u+\nu v)$ is singular so that  $(\mu:\nu:\mathbf 0:\mathbf 0)\in X^*\cap ((\mu:\nu)\times\p^n)$. By fixing $u,v$, by restricting to  $X\cap V(u\mathbf x-v\mathbf y)$ and by taking into account \eqref{dxi}, \eqref{dyj} and \eqref{extraeq} one deduces  that
$X^*=\p^1\times \widetilde Y\subset T$  (see also the next sections for more details). The other conclusions are now clear.
\end{proof}
\medskip

\begin{rem}{\rm To prove equality in \eqref{in1} one could have argued also in this way. Letting $\rho=\rk(\Hm(g))=\dim(Z_{Y^*})+1$, it is sufficient to prove that $\rk(\Hm(f))=\dim(Z_X)+1$ is equal to 
$\rho+3=\dim(S(L',\p^1\times Z_{Y^*}))+1$.

Clearly  $\Hm(f)$ is a $(2n+2)\times (2n+2)$ matrix,
whose rank can be computed in this way.  The $(2n)\times(2n)$ submatrix corresponding to the second partial derivatives with respect to the variables $x_i$ and $y_j$ has rank $\rho$ by \eqref{dxi} and \eqref{dyj}. The $2\times 2n$ submatrix of $\Hm(f)$ corresponding to the second partial derivatives with respect to the variables $(u,v)\times(x_i, y_j)$ increases the rank by $1$ by \eqref{du} and by \eqref{dv}. The $2\times 2$ submatrix corresponding to the second partial derivatives with respect to the variables $u,v$ increases the rank by $2$. In conclusion $\rk(\Hm(f))=\rho+1+2=\rho+3$ so that equality holds in \eqref{in1} (see also Remark \ref{rhohess}).}
\end{rem}
\medskip

\begin{rem}{\rm Obviously also other similar changes of variables, for example like $\mathbf z\to (u\mathbf x-v\mathbf y)^k$ with $k\geq 2$, will produce other interesting hypersurfaces
with vanishing hessian. Instead of pursuing further these generalizations, we prefer to focus on the geometrical properties of the previous examples
and on  the connections with the so called {\it Dual Cayley Trick}. 
}
\end{rem}

\subsection{Geometrical properties of $X=V(g(u\mathbf x-v\mathbf y))\subset\p^{2n+1}$, of $Z_X^*$ and of the associated polar map}\label{geopro}

Let notation be as above and suppose that $V(g)\subset\p^{n-1}$ has non-vanishing hessian. Then $$Z_X=S(L',W)=S(L',\p^1\times\p^{n-1})\subset(\p^{2n+1})^*$$ and
$$Z_X^*\simeq \p^1\times \p^{n-1}\subset V(u,v)=\langle Z_X^*\rangle\subset\p^{2n+1}$$
is the Segre variety defined in $V(u,v)$  by the equations:
$$
\rk
\left( \begin{array}{ccc}
x_1&\ldots&x_n\\
y_1&\ldots&y_n
\end{array}
\right)
=1.$$
In the terminology of \cite[Section 2.2]{CRS}, the linear space $$\Pi=V(u,v)=\p^{2n-1}\subset \Sing(X)$$ is the {\it core of $X$}.  
If $p=(0:0:\mathbf x:\mathbf y)\in V(u,v)$, then  we can take $(\mathbf x:\mathbf y)$ as coordinates on $V(u,v)$.
Let $L$ be the line of equations $\mathbf x=\mathbf 0=\mathbf y$,  let 
$$\xi =(u:v:\mathbf 0:\mathbf 0)\in L,$$
and let  
$$\Pi_\xi=\langle \Pi, \xi\rangle\subset\p^{2n+1}.$$ 
For a fixed $\xi$, the points of the hyperplane $\Pi_\xi$ can be parametrized by $(tu:tv:\mathbf x:\mathbf y)$ so that $(t:\mathbf x:\mathbf y)$ can
be taken as coordinates on $\Pi_\xi$.
Then $\Pi_\xi\cap X$ has the following equation in the hyperplane $\Pi_\xi$:
$$t^dg(u\mathbf x-v\mathbf y)=0$$
with $u,v$ fixed. Since $t=0$ is the equation of $\Pi\subset \Pi_\xi$, $\Pi_\xi\cap X$ contains $\Pi$ with multiplicity
$d$,
while $$V(g(u\mathbf x-v\mathbf y))\subset \Pi_\xi$$ is a cone with vertex a $\p^n$, which is
not contained in $\Pi$. The change of variable $u\mathbf x-v\mathbf y\mapsto \mathbf x$, $\mathbf y\mapsto \mathbf y$, $t\mapsto t$ shows that  the resulting equation
does not depend on the variables $\mathbf y$ and $t$, yielding that the vertex of the cone is the linear subspace of  $\Pi_\xi$ given by the $n$ linear equations $u\mathbf x-v\mathbf y=\mathbf 0$.

Varying the hyperplane in the pencil of hyperplanes through $V(u,v)$, the vertices  of the corresponding cones describe a Segre variety
$\p^1\times \p^n$, which is the dual of $T\subset(\p^{2n+1})^*$ and which cuts $V(u,v)$ along $Z^*_X$.

In particular, the series of examples constructed in this section is completely different from those known up to now,
which we shall simply call of {\it Gordan-Noether-Perazzo-Permutti-CRS type}. Indeed, in all these examples the intersection of the
linear spaces
$\Pi_\xi=\p^{c+1}$ through  the core  $\Pi=<Z_X^*>=\p^c$ with the hypersurface $X$ consists of the core with a suitable multiplicity and of a cone, whose vertex 
is a linear space contained in $\Pi$.

From the algebraic point of view this new phenomenon means that there does not exist a suitable linear change of coordinates
such that we can  {\it separate the variables in the equation via the core}. We recall that Gordan-Noether-Perazzo-Permutti-CRS type hypersurfaces in  $\p^4$
exhaust the list of hypersurfaces with vanishing hessian that are not  cones.

\begin{example}{\rm In  $\P^7$  with homogeneous coordinates $(u:v:x_1:x_2:x_3:y_1:y_2:y_3)$, let 
$$X = V((x_1u-y_1v)^2+(x_2u-y_2v)^2+(x_3u-y_3v)^2) \subset \P^7$$ 
and let $Y = V(z_1^2+z_2^2+z_3^2) \subset \P^2$ be the  self dual Fermat conic.
Letting the notation be as above, we have $X^* = \P^1 \times \widetilde Y\subset (\p^7)^*$. Let us remark that the construction of irreducible hypersurfaces 
of this kind starts from  $\P^7$. Specialisations of above examples  have interesting applications, see \cite[Example 2.3]{DeBondt2} and \cite{DeBondt1}.}
\end{example}

\begin{example}{\rm In  $\P^5$ with homogeneous coordinates $(u:v:x_1:x_2:y_1:y_2)$, let  
$$X = V((x_1u-y_1v)^2+(x_2u-y_2v)^2+u^4) \subset \P^5.$$
It is not difficult to see that $Z_X=V(\tilde{x}_1\tilde{y}_2-\tilde{x}_2\tilde{y}_1)\subset \P^5$ and that, taking into account the previous remarks, $X\subset\p^5$ is the first 
example of a hypersurface with vanishing hessian that is not a cone and that it is not of Gordan-Noether-Perazzo-Permutti-CRS type.}
 
\end{example}

\section{Hypersurfaces with vanishing hessian constructed from cones with vertex a $\p^{k-1}$ via Dual Cayley Trick}

\subsection{Resultants and discriminants} 
We  recall some well known facts on resultants and discriminants. A reference for most of
the properties listed below is \cite{GKZ}, see also \cite{Cox}.

Let $f_i(x_0,\ldots, x_N)$, $i\in\{0,\ldots,N\}$, be $N+1$ universal homogeneous polynomials
of degree $d_i\geq 1$. Then the {\it resultant of $f_0,\ldots, f_N$}, indicated by
$$\Res (f_0,\ldots, f_N),$$
is a polynomial in the coefficients of the $f_i$'s, which is homogeneous of degree
$d_0\cdots d_{j-1}d_{j+1}\cdots d_N$ in the variables corresponding to $f_j$
and which has degree 
\begin{equation}\label{degRes}
d_0\cdots d_N\sum_{i=0}^N\frac{1}{d_i}.
\end{equation}

The polynomial $\Res(f_0,\ldots,f_N)$  has the following property:
given  homogeneous   polynomials $g_0,\ldots, g_N\in\mathbb K[x_0,\ldots, x_N]$ with $\deg(g_i)=d_i$,  the value of $\Res (f_0,\ldots, f_N)$ on the coefficients 
of $g_0,\ldots, g_N$ is zero if and only if $g_0=\ldots=g_N=0$ has a non-zero solution in $\mathbb K^{N+1}$ (or equivalently
$V(g_0)\cap\ldots\cap V(g_N)\neq\emptyset$ where $V(g_i)\subset\p^N_\mathbb K$ is the projective hypersurface defined by $g_i$), see \cite{GKZ}.

For a universal  $f\in\mathbb K[x_0,\ldots,x_N]_d$, let 
$$\Delta_{N,d}=\Res(\frac{\partial f}{\partial x_0},\ldots, \frac{\partial f}{\partial x_N}),$$
which is a homogeneous polynomial of degree $(N+1)(d-1)^N$ in the ${N+d}\choose d$ coefficients of
the universal $f \in\mathbb K[x_0,\ldots,x_N]_d$. By the previous geometrical} property of the resultant we deduce
that  the (geometrically irreducible) hypersurface 
$V(\Delta_{N,d})\subset\p(\mathbb K[x_0,\ldots,x_N]_d)$, called the  {\it discriminant hypersurface}, is well defined and it describes the locus of  singular projective  hypersurfaces of degree $d$.
\medskip

 \subsection{Dual varieties of rational normal scrolls surfaces and some explicit examples of hypersurfaces with vanishing hessian}\label{dualscrolls}
 
 Let $1\leq a\leq b$ be integers and let $$S(a,b)\subset\p^{a+b+1}$$ be a rational normal scroll of degree $d=a+b$.
 The surface $S(a,b)\subset\p^{a+b+1}$ is  
smooth and projectively generated by a rational normal curve $C_a=\nu_a(\p^1)\subset\p^a$ and a rational normal curve $C_b=\nu_b(\p^1)\subset\p^b$
with $\langle C_a\rangle\cap\langle C_b\rangle=\emptyset$, by taking the union of the lines $\langle \nu_a(p),\nu_b(p)\rangle$, $p\in\p^1$.

We shall choose coordinates $(x_0:\cdots:x_a:y_0:\cdots:y_b)$ on $\p^{a+b+1}$ such that $V(y_0,\ldots,y_b)=\langle C_a\rangle$ and such that
$V(x_0,\ldots, x_a)=\langle C_b\rangle$. Accordingly, $C_a\subset\p^{a+b+1}$ has parametrization $(s^a:s^{a-1}t:\cdots:st^{a-1}:t^a:0:0:\cdots:0)$
and $C_b\subset\p^{a+b+1}$ has parametrization $(0:0:\cdots:0:s^b:s^{b-1}t:\cdots:st^{b-1}:t^b)$. 

The following result is well known, see for example \cite[Example 3.6]{GKZ}, and it shows the existence of a lot of significative examples of hypersurfaces with vanishing hessian, not cones.
Special projections of such examples in $\P^4$ produce examples of Gordan-Noether-Perazzo-Permutti-CRS hypersurfaces  (see \cite{CRS}). 

\begin{theorem}\label{dualSab} Let notation be as above and let $(w_0:\cdots:w_a:z_0:\cdots:z_b)$ be dual coordinates
to $(x_0:\cdots:x_a:y_0:\cdots:y_b).$ Then
$$S(a,b)^*=V(\Res(f,g))\subset (\p^{a+b+1})^*$$ is a hypersurface of degree $a+b$,
where $f=w_0s^a+w_1s^{a-1}t+\cdots +w_at^a\in\mathbb K[s,t]_a$
is a general binary form of degree $a$ and where $g=z_0s^b+z_1s^{b-1}t+\cdots+z_bt^b
\in\mathbb K[s,t]_b$
is a general binary form of degree $b$.

In particular,
$$S(1,b)^*=V((-w_1)^bz_0+(-w_1)^{b-1}w_0z_1+\cdots-w_1w_0^{b-1}+w_0^bz_b)\subset(\p^{b+2})^*$$
is a hypersurface of degree $b+1$, which is not a cone and which for $b\geq 2$ has vanishing hessian.
\end{theorem}
\begin{proof}
Let notation be as above. To calculate the parametric equations of 
the tangent plane to $S(a,b)$ at a general point $q=\lambda\nu_a(s,t)+\mu\nu_b(s,t)$ 
with $(s:t)\in\p^1$ we shall suppose
$(s:t)=(1:t)$ and $(\lambda:\mu)=(1:\mu)$. 
In particular we can suppose  $x_0=1$ and $y_0=\mu$. 
Thus 
the projective tangent space $T_qS(a,b)$ is spanned by the rows of the following matrix:
$$
\left( \begin{array}{cccccccc}
1&t&\ldots&t^a&\mu&\mu t&\ldots&\mu t^b\\
0&1&\ldots&at^{a-1}&0&\mu&\ldots&\mu b t^{b-1}\\
0&0&\ldots&0&1&t&\ldots&t^b
\end{array}
\right),
$$
and hence it is 
also spanned by the rows of the matrix
\begin{equation}\label{parTq}
\left( \begin{array}{cccccccc}
1&t&\ldots&t^a&0&0&\ldots&0\\
0&1&\ldots&at^{a-1}&0&\mu&\ldots&\mu b t^{b-1}\\
0&0&\ldots&0&1&t&\ldots&t^b
\end{array}
\right).
\end{equation}
If $(w_0:\cdots:w_a:z_0:\cdots:z_b)$ are dual coordinates on $(\p^{a+b+1})^*$, we get that
a point of $(\p^{a+b+1})^*$ belongs to $S(a,b)^*$ if and only if 
$$
\left\lbrace \begin{array}{lcl}
w_0+w_1t+\cdots + w_at^a&=&0\\
w_1+2w_2 t+\cdots+aw_{a}t^{a-1}+\cdots+\mu z_1+2\mu tz_2+\cdots+b\mu t^{b-1}z_b&=&0\\
z_0+tz_1+\cdots+t^bz_b&=&0 
\end{array}\right.
$$
has a solution $(t,\mu)$. Since the second equation is linear in $\mu$, this happens if
and only if
$$
\left\lbrace \begin{array}{lcl}
w_0s^a+w_1s^{a-1}t+\cdots + w_at^a&=&0\\
z_0s^b+z_1s^{b-1}t+\cdots+z_bt^b&=&0 
\end{array}\right.
$$
has a solution $(s,t)\neq (0,0)$. In conclusion, the equation of $S(a,b)^*$ is the resultant of two general homogeneous forms 
of degree $a$ and degree $b$ in the variables $(s,t)$. Therefore, $S(a,b)^*$ is a hypersurface of degree $d=a+b$ by \eqref{degRes},
whose equation can be explicitly written (for example by using Sylvester Formula).

For $a=1$, the first equation is $sw_0+tw_1=0$, whose roots are $(-w_1,w_0)$. Thus the equation of $S(1,b)^*$ is obtained
by imposing that $(-w_1,w_0)$ is a solution of the second equation, that is
$$S(1,b)^*=V((-w_1)^bz_0+(-w_1)^{b-1}w_0z_1+\cdots+w_0^bz_b)\subset(\p^{b+2})^*.$$

If $b>1$, then  the partial derivatives of the equation of $S(1,b)^*$ with respect to $z_i$ are algebraically dependent so that
the hypersurface $S(1,b)^*\subset(\p^{b+2})^*$ has vanishing hessian and  is not a cone.
\end{proof}
\medskip

The previous analysis admits obvious generalizations we shall only mention without proofs. The surface $S(a,b)\subset\p^{a+b+1}$ can be seen as
the embedding of $\p(\O_{\p^1}(a)\oplus\O_{\p^1}(b))$ into $\p^{a+b+1}=\p( \Hm^0(\O_{\p^1}(a))\oplus\Hm^0(\O_{\p^1}(b)))$
by the tautological line bundle $\O(1)$. Thus, letting $r\geq 1$ and letting 
$$\p^{N(a_0,\ldots,a_r)}=\p(\Hm^0(\O_{\p^r}(a_0))\oplus\ldots \Hm^0(\O_{\p^r}(a_r))),$$
we shall suppose $1\leq a_0\leq\ldots\leq a_r$ and consider 
$$X(a_0,\ldots, a_r)=\p(\O_{\p^r}(a_0)\oplus\ldots \oplus\O_{\p^r}(a_r))\subset\p^{N(a_0,\ldots,a_r)}$$
embedded by the tautological line bundle $\O(1)$. This smooth manifold
is a $\p^r$-bundle over
$\p^r$, which is projectively generated by the $r+1$ varieties  $\nu_{a_i}(\p^r)$ lying in disjoint linear subspaces of 
$\p^{N(a_0,\ldots,a_r)}.$

The same calculations used in the proof of Theorem \ref{dualSab} above 
prove that 
$$X(a_0,\ldots, a_r)^*=V(\Res(f_0,\ldots, f_r)),$$
where $f_i$ is a generic polynomial of degree $a_i$ for $i=0,\ldots, r$. Moreover,  
$$\deg(X(a_0,\ldots, a_r)^*)=a_0\cdots a_r\sum_{i=0}^r\frac{1}{a_i}$$
 by \eqref{degRes}.
 
In particular, if $a_0=\ldots=a_{r-1}=1$ and if $a_r=a\geq 2$, then $$X(1,\ldots, 1, a)^*\subset(\p^{N(1,\ldots,1,a)})^*$$ is a hypersurface 
of degree $r\cdot a+1$ with vanishing hessian which is  not
a cone.
\medskip

\begin{rem}{\rm
These are the first instances  of a general method, which  has been dubbed
{\it the Cayley Trick for mixed resultants} in \cite[Ch. 3, sections 2, 3, 4]{GKZ}, for calculating the explicit equations of dual varieties  of $\p^r$-bundles of the form $\p(\mathcal E)$  embedded by 
the tautological line bundle $\O(1)$ with  $\mathcal E$ a very ample rank $r+1$ locally free sheaf over an irreducible projective variety
$X$ of dimension $r$.}
\end{rem}

We now introduce  and apply the
classical {\it Cayley Trick} and its dual variant, the so called {\it Dual Cayley Trick} to calculate explicitly the equations of some dual varieties.

\subsection{Cayley Trick}\label{notCayley}  Let $X\subset\p^N=\p(V)$ be an irreducible non-degenerate variety of dimension $n\geq 1$.
Let $\mathbb G(r,\p(V))$ denote the Grassmann variety of $r$-dimensional projective subspaces of $\p(V)$. If $L=\p(U)\subset\p(V)$
has dimension $r\geq 0$, then $L^\perp=\p(\Ann(U))\subset\p(V^*)$ has dimension $N-r-1$ and we have a natural isomorphism
$\mathbb G(r,\p(V))\simeq\mathbb G (N-r-1,\p(V^*)),$ defined by sending $[L]$ to $[L^\perp]$. We have two natural rational maps:
$$q:\p(\mathbb K^{r+1}\otimes V)\map \mathbb G(r,\p(V))$$
and
$$p:\p(\mathbb K^{N-r}\otimes V^*)\map \mathbb G(r,\p(V)),$$
corresponding, respectively,  to the parametric equations and  to the
cartesian equations of a subspace $L=\p(U)\subset \p(V)$. The rational maps
are defined on the open sets of elements of maximal rank and on these
open sets they are the quotient maps of the natural action via left multiplication of the group
of invertible matrices. 
\medskip

Let $X\subset\p^{N}=\p(V)$ be as above and let $e=\deg(X)\geq 2$. 
Following \cite{GKZ}, let
$$Z(X)=\{[L]\in\mathbb G(N-n-1,\p(V)):\; L\cap X\neq\emptyset\}\subset\mathbb G(N-n-1,\p(V))$$
be the {\it associated hypersurface of $X$}. Indeed,  $$\codim(Z(X), \mathbb G(N-n-1,\p(V))=1,$$
see \cite[Proposition 2.2, Chap. 3]{GKZ}, and $Z(X)$ is given by a homogeneous element of degree $e=\deg(X)$ in the homogeneous coordinate ring of 
$$\mathbb G(N-n-1,\p(V))\subset
\p(\Lambda^{N-n}V),$$ defined modulo Pl\" ucker relations and dubbed the {\it Chow form of $X$}.

\begin{example}{\rm
Let $f_0,\ldots, f_N\in\mathbb K[x_0,\ldots, x_N]_d$ be homogeneous forms of degree $d\geq 1$.
Then $\Res (f_0,\ldots, f_N)$ is  a homogeneous polynomial of degree $(N+1)d^N$ in 
the $(N+1)\times {{N+d}\choose {d}}$ variables, which are the coefficients of the universal $f_i$'s
or equivalently the homogeneous coordinates  on $\p(\mathbb K[x_0,\ldots, x_N]_d).$

Under this assumption, if $[a_{i,j}]\in GL_{N+1}(\mathbb K)$ and if $$h_i=\sum_{j=0}^N a_{i,j}g_j,$$ then one proves that
\begin{equation}\label{invRes}
\Res(h_0,\ldots,h_N)=\det([a_{i,j}])^{d^N}\Res(g_0,\ldots, g_N).
\end{equation}

In particular, $\Res(f_0,\ldots, f_N)$ is an invariant for  the action by left multiplication of $SL_{N+1}(\mathbb K)$ on the set
of $(N+1)\times {{N+d}\choose {d}}$ matrices with entries in the coefficients of $f_0,\ldots, f_N$. 

By the  {\it First
Theorem of Invariant Theory}  the polynomial $\Res(f_0,\ldots,f_N)$ can be written as a polynomial
of degree $d^N$ in the $(N+1)\times (N+1)$ minors of the $(N+1)\times {{N+d} \choose {d}}$ matrix associated to $\{f_0,\ldots, f_N\}$,
that is in the Pl\" ucker coordinates of the matrix. This polynomial is the (dual) Chow form of $\nu_d(\p(V))\subset\p(S^d(V))=\p^{N(d)}$,
defined modulo Pl\" ucker coordinates, and its  restriction  to $\mathbb G(N(d)-N-1,\p(S^d(V)))$ defines
$Z(\nu_d(\p(V)))$ by the geometrical interpretation of the resultant.

Letting $$p:\p(\mathbb K^{N+1}\otimes (S^dV)^*)\map \mathbb G(N(d)-N-1,\p(S^d(V)))$$ be the natural map
defined above, we deduce
$$V(\Res (f_0,\ldots, f_N))=\overline{p^{-1}(Z(\nu_d(\p(V)))}.$$

Let $y_0,\ldots, y_N$ be other variables and let $$y_0f_0+\cdots+y_Nf_N\in\mathbb K[x_0,\ldots, x_N, y_0,\ldots, y_N]_{d+1},$$
which is also a bihomogeneous polynomial of bidegree $(d,1)$. Then the  classical {\it Cayley Trick} is the formula:
\begin{equation}\label{Cayleytrick}
\mathrm{Res}(f_0,\ldots, f_N)=\Delta(y_0f_0+\cdots+y_Nf_N),
\end{equation}
a useful  remark which dates back to Cayley.

The geometrical translation of the Cayley Trick is the following: if  
$$\p(\mathbb K^{N+1})\times \nu_d(\p(V))\subset \p^{(N+1)(N(d)+1)-1}=\p(\mathbb K^{N+1}\otimes S^d(V))$$
is the Segre embedding of $\p^N\times\nu_d(\p^N)$, 
then
$$(\p^N\times \nu_d(\p(V)))^*=\overline{p^{-1}(Z(\nu_d(\p(V))))}.$$
Indeed, formula \eqref{Cayleytrick} says that if a hyperplane $H\subset\p(\mathbb K^{N+1}\otimes S^d(V))$ contains
$\p^N\times p$, $p\in\nu_d(\p^N)$ (the condition on the left), then there exists $q\in \p^N$ such that $q\times T_p\nu_d(\p^N)\subset H$
so that $H$ is tangent to $\p^N\times \nu_d(\p^N)$ at $(q,p)$ (the condition on the right), yielding  $[H]\in(\p^n\times\nu_d(\p^N))^*$.}
\end{example}
\medskip

Nowadays the geometrical version of the Cayley Trick has been generalized by Gelfand, Weyman and Zelevinsky  to arbitrary irreducible varieties.
\medskip

\begin{theorem}\label{Cayleythm} {\rm (Cayley Trick, \cite[Theorem 2.7, Chap. 3]{GKZ})} Let $X\subset\p^N=\p(V)$ be an irreducible non-degenerate variety of dimension $n\geq 1$ and let
$\p^n\times X\subset\p(\mathbb K^{n+1}\otimes V)$ be the Segre embedding. Then
\begin{equation}\label{CayleyX}
(\p^n\times X)^*=\overline{p^{-1}(Z(X))},
\end{equation}
where $p:\p(\mathbb K^{n+1}\otimes V^*)\map \mathbb G(N-n-1,\p(V))$ is the quotient map corresponding to cartesian
equations of linear subspaces of dimension $N-n-1$ of $\p(V)$.
\end{theorem}
\medskip

\subsection{Dual Cayley Trick}
We now present the  so called {\it Dual Cayley trick},  introduced  by Weyman and Zelevinsky in \cite{WZ}, see also \cite{Kohn}.

Let $\widetilde Y\subset\p^N=\p(V)$ be an irreducible variety of dimension $n\geq 1$ such that $\widetilde Y^*\subset\p(V^*)$ has dimension $N-1-r$. Let $\p^r=\p(T)$
and let 
$$\p^r\times \widetilde Y\subset\p^{(r+1)(N+1)-1}=\p(T\otimes V)$$
be the Segre embedding of $\p(T)\times \widetilde Y$. Then by \cite[Corollary 3.3]{WZ} the dual variety
$(\p^r\times \widetilde Y)^*\subset(\p^{(r+1)(N+1)-1})^*$ is a hypersurface 
$$V(f)\subset (\p^{(r+1)(N+1)-1})^*=\p((T\otimes V)^*)=\p(T^*\otimes V^*),$$
which can be computed in this way.

Let $$q:(\p^{(r+1)(N+1)-1})^*=\p(T^*\otimes V^*)\map \mathbb G(r,\p(V^*))$$
be the natural rational map defined above and corresponding to parametric equations of linear subspaces of dimension $r$ of $\p(V^*)$. This map, in the natural coordinates, sends 
 a rank $r+1$ matrix $\mathbf X\in\mathbb K^{r+1,N+1}$ to its Pl\" ucker coordinates.  It is thus given by forms of degree $r+1$ in the natural coordinates.

Let $$Z(\widetilde Y^*)\subset  \mathbb G(r,\p(V^*))$$
be the  associated hypersurface of $\widetilde Y^*$.
Then \cite[Proposition 4.2.b]{WZ} yields the following  formula:
\begin{equation}\label{WZformula}
(\p(T)\times \widetilde Y)^*=\overline{q^{-1}(Z(\widetilde Y^*))}.
\end{equation}

\subsection{Hypersurfaces with vanishing hessian constructed from cones with vertex a $\p^{r-1}$}
We now apply the Dual Cayley Trick to generalize part (ii) of  Theorem \ref{codimension}, giving also a different and more theoretical
proof of this result.

Let $r\geq 1$ be an integer and let $\widetilde Y\subset\p^N=\p(V)$ be a cone with vertex a $\p^{r-1}=\p(U)$ over a non-degenerate
variety
$Y\subset\p^{N-r}=\p(W)$ without dual defect, that is $Y^*=V(g)\subset(\p^{N-r})^*=\p(W^*)$ is a hypersurface
with $g=g(z_0,\ldots, z_{N-r})\in \mathbb K[z_0,\ldots, z_{N-r}]_d$ for some $d\geq 2$. 
By definition $V=U\oplus W$.

Let $\p^r=\p(T)$ and let 
$$\p(T)\times \widetilde Y\subset\p(T)\times\p(V)\subset\p(T\otimes V)=\p^{(r+1)(N+1)-1},$$
be the Segre embedding of $\p^r\times \widetilde Y$.

We can identify points of $(\p^{(r+1)(N+1)-1})^*=\p((T\otimes V)^*)$ with matrices  
\begin{equation}\label{minors}
\mathbf X=\left[
\begin{tabular}{c|c}
A'&B'
\end{tabular}
\right],
\end{equation}
where $A'\in\mathbb K^{r+1,N-r+1}$ and $B'\in \mathbb K^{r+1,r}$. This decomposition corresponds
to the natural decomposition
$$\Hom(T,V^*) \simeq \Hom(T,W^*)\oplus\Hom(T,U^*)\simeq (T^*\otimes W^*)\oplus (T^*\otimes U^*).$$

The linear system of equations $B'=0_{(r+1)\times r}$ defines
the linear span $\p(\Hom(T, W^*))=\p(T^*\otimes W^*)$ of  the Segre variety $$R=\p(T^*)\times\p(W^*)\subset(\p^{(r+1)(N+1)-1})^*.$$

There are exactly $N-r+1$ minors $B'_j$ of order $r+1$, obtained from $\mathbf X$ as in \eqref{minors} by adding the $j^{\th}$
column of $A'$ to $B'$, $j=0,\ldots, N-r$. Define
$$\phi_1:(\p^{(r+1)(N+1)-1})^*\map \p^{(r+1)r-1}=\p(\Hom(T,U^*)),$$
by 
$$\phi_1(\mathbf X)=B' ,$$
and
$$\phi_2:(\p^{(r+1)(N+1)-1})^*\map (\p^{N-r})^*=\p(W^*),$$
by 
$$\phi_2(\mathbf X)=(B'_0:\cdots:B'_{N-r})\in (\p^{N-r})^*.$$

The point $\phi_2(\mathbf X)$ is the intersection of $\p(W^*)$ with the $r$-dimensional linear subspace of $\p(V^*)$
corresponding to $\mathbf X$.
\medskip

\begin{theorem}\label{codimensionk} Let the hypothesis and the notation be as above,  let $$Z_{Y^*}=\overline{\nabla_g(\p(W^*))}\subseteq\p(W)=\p^{N-r},$$  let
$$\p^r\times Z_{Y^*} \subset\p(T\otimes W)$$ be the Segre embedding, let $X=V(g(B'_0,\ldots,B'_{N-r}))\subset\p(T^*\otimes V^*)$ and let
$f=g(B'_0,\ldots,B'_{N-r}).$ Then:
\medskip
\begin{enumerate}
\item[(i)] $(\p^r\times\widetilde Y)^*=V(g(B'_0,\ldots,B'_{N-r}))\subset\p(T^*\otimes V^*)$;
\medskip
\item[(ii)] $Z_X=\overline{\nabla_f(\p(T^*\otimes V^*))}\subseteq S(\p(T\otimes U),\p(T)\times Z_{Y^*})\subset\p(T\otimes V)$ and
$X$ has vanishing hessian.
\medskip
\end{enumerate}
\end{theorem}
\begin{proof} One can argue as in the proof of Theorem \ref{codimension} but we prefer
to deduce this more general result from \eqref{WZformula},  specializing the Dual Cayley Trick to our situation.

Recall that in this case  $\widetilde Y^*=Y^*\subset (\p^{N-r})^*=\p(W^*)\subset\p(V^*)$. Moreover,   we can define a rational map 
$$\psi:\mathbb G(r,\p(V^*))\map(\p^{N-r})^*=\p(W^*),$$
by
$$\psi([L])=[L\cap \p(W^*)].$$
The map $\psi$ is not defined along $\mathbb G(r, \p(W^*))\subset\mathbb G(r,\p(V^*))$ and   along
the Schubert cycles given by the  $[L]$'s  such that $\dim(L\cap \p(W^*))>0$.

Since, by hypothesis, $\widetilde Y^*=Y^*=V(g)$ is a hypersurface in $\p(W^*)$, it follows that 
$$\overline{\psi^{-1}(Y^*)}=Z(\widetilde Y^*).$$

Using the coordinates introduced above, we deduce that
\begin{equation}\label{key}
\phi_2=\psi\circ q, 
\end{equation}
where $q:\p(T^*\otimes V^*)\map \mathbb G(r,\p(V^*))$ is the natural rational map considered above.

Then \eqref{WZformula} gives
$$(\p(T)\times \widetilde Y)^*=\overline{q^{-1}(Z(\widetilde Y^*))}.$$
Using \eqref{key} we get

$$\overline{q^{-1}(Z(\widetilde Y^*))}=\overline{q^{-1}(\psi^{-1}(\widetilde Y^*))}=\overline{\phi_2^{-1}(Y^*)}=V(g(B'_0,\ldots,B'_{N-r}))=X.$$

Let notation be as in  \eqref{minors}, let $a'_{i,j}$, with $i=0,\ldots, r$ and with $j=0,\ldots, N-r$, be the homogeneous coordinates corresponding to $A'$
and let $b'_{i,k}$ with $k=1,\ldots, r$ be the coordinates corresponding to $B'$. By definition of $B'_j$, $j=0,\ldots, N-r$, we deduce from  Laplace Formula applied to the first column of $B'_j$:
$$B'_j=\sum_{i=0}^r (-1)^i a_{i,j}C_i$$
yielding
$$\frac{\partial B'_j}{\partial a_{i,j}}=(-1)^i C_i.$$
Moreover, for $m\neq j$, we have
$$\frac{\partial B'_j}{\partial a_{i,m}}=0.$$

From 
\begin{equation}\label{dfij}
\frac{\partial f}{\partial a_{i,j}}=(-1)^iC_i\frac{\partial g}{\partial z_j}(B'_0,\ldots, B'_{N-r})
\end{equation}
we deduce that, for every $i\neq k$ and for every $l\neq m$,
\begin{equation}\label{ik-ki}
\frac{\partial f}{\partial a_{i,l}}\frac{\partial f}{\partial a_{k,m}}-\frac{\partial f}{\partial a_{i,m}}\frac{\partial f}{\partial a_{k,l}}=0.
\end{equation}

Thus $X\subset\p(T^*\otimes V^*)$ has vanishing hessian since the partial derivates of $f$ are algebraically dependent and more precisely
\begin{equation}\label{finclusion}
Z_X=\overline{\nabla_f(\p(T^*\otimes V^*))}\subseteq S(\p(T\otimes U),\p(T)\times Z_{Y^*})\subset\p(T\otimes V).
\end{equation}
\end{proof}

Let us remark that for $r=1$ the base locus of $\psi$ is exactly $\mathbb G(1,(\p(W^*))$ since a line cuts $\p(W^*)$ in one point if and only
if it is not contained in $\p(W^*)$. Thus in this case the expression of $\phi_2$ is, modulo the obvious identifications, that given in \eqref{phi2}.
\medskip

\begin{rem}\label{rhohess}{\rm 
One could ask if equality holds in \eqref{finclusion}. Letting  $\rho=\rk(\Hm (g))=\dim(Z_{Y^*})+1$, it would be sufficient (indeed equivalent) to prove
that $$\rk(\Hm(f))-1=\dim(Z_X)=\dim(S(\p(T\otimes U),\p(T)\times Z_{Y^*}))=r(r+1)+r+\rho-1,$$
i.e. that $\rk(\Hm(f))=r(r+1)+r+\rho.$

Since this analysis is quite delicate (and also intricate) we preferred to skip the details and to concentrate  on  the interesting connections with the Dual Cayley Trick in order to produce the new  examples, which  generalize to arbitrary $r\geq 2$  the case $r=1$ considered in part (ii) of Theorem \ref{codimension}. Last but not least, we point out that  Theorem \ref{codimension} is sufficient to construct examples of hypersurfaces with vanishing hessian with $\codim(X^*,Z_X)$ arbitrary large. Assuming equality in \eqref{finclusion},  one would deduce $\codim(X^*,Z_X)=\codim(Y, Z_{Y^*})+r^2$ and there is no 
advantage in solving the previous equation instead of the simpler $\codim(X^*, Z_X)=\codim(Y, Z_{Y^*})+1$. }
\end{rem}

\subsection{The dual variety of $\p^1\times Y\subset\p^{2n+3}$ with $Y=V(f)\subset\p^{n+1}$ an irreducible hypersurface}

Let $X\subset\p^{N}=\p(V)$ be an irreducible projective variety of dimension $n=\dim(X)$ and degree  $e\geq 2$. Let $\p^{n-i}=\p(T)$,
$i\in\{0,\ldots, n\}$
and let 
$$\p^{n-i}\times X\subset\p^{(n-i+1)(N+1)+1}=\p(T\otimes V)$$
be the Segre embedding of $\p(T)\times X$. 
Let $Z_i(X)\subset \mathbb G(N-n+i-1,\p(V))$ 
be the {\it $i^{\th}$ higher associated variety of $X$} in the sense of \cite{GKZ}, i.e. it is
 the closure of the set
$$\{[L]\in\mathbb G(N-n+i-1,\p(V)): \exists\; x\in X_{\reg}: x\in L, \dim(L\cap T_xX)\geq i\} .$$
Clearly $Z_0(X)=Z(X)$ is the Chow hypersurface of $X$ 
and $$Z_n(X)=X^*\subset \p(V^*)=\mathbb G(N-1,\p(V)).$$
Let us recall that $Z_i(X)$ is a hypersurface in $\mathbb G(N-n+i-1,\p(V))$ if and only if $i\leq n-\codim(X^*)+1$, see \cite{GKZ, WZ} and
\cite{Kohn}.
In particular $Z_{n-1}(X)$ is a hypersurface if and only if  $\codim(X^*)\in\{1,2\}$.

Let $$p:(\p^{(n-i+1)(N+1)-1})^*=\p(T^*\otimes V^*)\map \mathbb G(N-n+i-1,\p(V))$$
be the rational map defined in Subsection \ref{notCayley} by considering the cartesian
equations of a linear subspace of $\p(V)$.
We defined also the rational map 
$$q:(\p^{(n-i+1)(N+1)-1})^*=\p(T^*\otimes V^*)\map \mathbb G(n-i,\p(V^*))$$
associated to the representation of a linear subspace of $\p(V)^*$ via parametric equations.
Then \cite[Proposition 4.2.a]{WZ}, see also \cite{Kohn}, yields the following  formulas:
\begin{equation}\label{formulaiX}
(\p^{n-i}\times X)^*=\overline{p^{-1}(Z_i(X))},
\end{equation}

\begin{equation}\label{formulaiXdual}
(\p^{n-i}\times X)^*=\overline{q^{-1}(Z_{n-\codim(X^*)-i+1}(X^*))}.
\end{equation}
Suppose now that $Y\subset\p^{n+1}=\p(V)$ is a hypersurface such that $Y^*\subset\p(V^*)$
is a hypersurface. Then the previous formulas give
\begin{equation}\label{p1hyp}
(\p^1\times Y)^*=\overline{p^{-1}(Z_{n-1}(Y))}=\overline{q^{-1}(Z_1(Y^*))}.
\end{equation}

Let $Y\subset\p(V)=\p^{n+1}$ be an irreducible hypersurface of degree at least two, then $Z_1(Y)\subset \mathbb G(1,\p(V))$ is a hypersurface
parametrizing the tangent lines to $Y$. It is given by a polynomial
 in the Pl\" ucker coordinates of $\mathbb G(1,\p(V))$. 
 To determine the degree $e$ of this polynomial let us 
 remark that a general line
$L\subset  \mathbb G(1,\p(V))$  consists of the lines $l\subset \p(V)$ passing through a general point
$p\in\p(V)$ and contained in a general plane $\Pi\subset\p(V)$ with $p\in\Pi$. 
Then $e=\#(L\cap Z_1(Y))$ equals the number of tangent lines to $Y\cap \Pi$ passing through $p$,
that is 
$e=\deg((Y\cap \Pi)^*).$
\medskip

\begin{example}\label{Z1Fermat}{\rm Let $Y=V(x_0^2+\cdots+x_{n+1}^2)\subset\p^{n+1}$ be the Fermat quadric hypersurface.
Then  $Z_1(Y)\subset\mathbb G(1,\p^{n+1})$ is a hypersurface of degree 2 whose equation is
quadratic and, modulo Pl\"ucker relations, is the quadratic Fermat in the Pl\"ucker coordinates, that is 
$$Z_1(V(x_0^2+\cdots+x_{n+1}^2))=V(\sum_{0\leq i<j\leq n+1}p_{i,j}^2)\subset \mathbb G(1,\p^{n+1}).$$
Let $f(x_0,\dots,x_{n+1})=x_0^2+x_1^2+\cdots+x_{n+1}^2$, then 
 $Y^*=V(f(y_0,\ldots,y_{n+1}))\subset(\p^{n+1})^*$. Let  $\mathbf a=(a_0:\cdots :a_{n+1})$ and $\mathbf b=(b_0:\cdots:b_{n+1})$
and let $(\mathbf a:\mathbf b)$ the natural coordinates on $\p^{2n+3}$ in such a way that the Pl\" ucker
coordinates $q_{i,j}$ of a matrix whose first row is $\mathbf a$ and whose second row is $\mathbf b$ are
$q_{i,j}=a_ib_j-a_jb_i$ for $0\leq i<j\leq n+1$. Then
$$Z_1(V(y_0^2+\cdots+y_{n+1}^2))=V(\sum_{0\leq i<j\leq n+1}q_{i,j}^2)\subset \mathbb G(1,(\p^{n+1})^*)$$
and
$$(\p^1\times Y)^*=\overline{q^{-1}(Z_1(Y^*))}=V(\sum_{0\leq i<j\leq n+1}(a_ib_j-a_jb_i)^2) \subset(\p^{2n+3})^*.$$
By Lagrange's Identity
$$\sum_{0\leq i<j\leq n+1}(a_ib_j-a_jb_i)^2=||\mathbf a ||^2\cdot ||\mathbf b ||^2-(\mathbf a\bullet \mathbf b)^2.$$
Thus  the dual of $\p^1\times V(x_0^2+\cdots+x_{n+1}^2)\subset \p^{2n+3}$ has a  {\it Cauchy-Schwartz} equation:
$$(\p^1\times V(x_0^2+\cdots+x_{n+1}^2))^*=V(||\mathbf a ||^2\cdot ||\mathbf b ||^2-(\mathbf a\bullet \mathbf b)^2).$$
This is a {\it homaloidal polynomial}, that is the associated polar map is a Cremona transformation of $\p^{2n+3}$.
In particular, the hessian of this quartic polynomial is different from zero and one can verify that it has a unique irreducible factor equal
to the polynomial itself.}
\end{example}
\medskip

\section{Duals of internal projections of Scorza Varieties from a point have vanishing hessian}

The series of varieties $$\p^n\times \p^n\subset \p^{n^2+2n} \text{ (Segre embedded, $n\geq 2$),}$$

$$\nu_2(\p^n)\subset\p^{\frac{n^2+3n}{2}} \text{ (quadratic Veronese embedding, $n\geq 2$),}$$

$$\mathbb G(1,\p^{2m+1})\subset \p^{2m^2+3m}  \text{ (Pl\" ucker embedding, $m\geq 2$)}$$
together with the Severi variety $E^{16}\subset\p^{26}$ are the so called {\it Scorza varieties.} 
These varieties and their duals have a uniform description via linear algebra and via the theory of determinantal varieties
we now briefly recall. 

\subsection{Generic determinantal Scorza varieties}

 Let $$\p^{n^2+2n}=\p(M_{(n+1)\times (n+1)}(\mathbb K)).$$
We shall indicate the generic matrix in $M_{(n+1)\times (n+1)}(\mathbb K)$
by
$$X=[x_{i,j}],$$
$i,j=0,\ldots, n$ and, by abusing notation, we shall also consider $$(x_{0,0}:\cdots:x_{n,n})$$ as homogeneous
coordinates on $\p^{n^2+2n}$. Analogously,  we shall indicate by $Y=[y_{i,j}]$ the matrices in the dual
space $\p((M_{(n+1)\times (n+1)}(\mathbb K))^*)$ in such a way that $(y_{0,0}:\cdots: y_{n,n})$  are homogeneous
coordinates dual to the previous ones.

For every $r=1,\ldots, n+1$ we can define the variety 
$$X_r=\{[X]\;:\;\rk(X)\leq r\}\subset \p(M_{(n+1)\times (n+1)}(\mathbb K));$$
the variety $Y_r\subset \p((M_{(n+1)\times (n+1)}(\mathbb K))^*)$ is
defined in the same way. 
With this notation we have $$X_1=\p^n\times\p^n\subset\p^{n^2+2n}$$
Segre embedded and  $$X_n=V(\det(X))\subset\p^{n^2+2n}$$ is a hypersurface of degree $n+1$. For simplicity,
let 
$$f=\det(X)\in\mathbb K[x_{i,j}]_{n+1}$$
and consider
$$\nabla_f:\p(M_{(n+1)\times (n+1)}(\mathbb K))\map \p((M_{(n+1)\times (n+1)}(\mathbb K))^*).$$

Letting
$$X^\#\in M_{(n+1)\times (n+1)}(\mathbb K)$$ be the matrix
defined by the Laplace formula:
\begin{equation}\label{Lapform}
X\cdot X^\#=\det(X)\cdot I_{(n+1)\times (n+1)}=X^\#\cdot X,
\end{equation}
the identity
\medskip
\begin{equation}\label{adjform}
(X^\#)^\#=\det(X)^{n-1}\cdot X
\end{equation}
shows that   $\nabla_f([X])=[X^\#]^t$ is  birational outside $X_n$.
The map is not  defined along  $X_{n-1}=\Sing(X_n)$ and its  ramification divisor  is given by the determinant of the Jacobian matrix of $\nabla_f$, which is $\hess(f)$. Since $f=\det(X)$ is an irreducible polynomial,
we deduce from \eqref{adjform} that the ramification divisor of $\nabla_f$ is supported on $X_n$, yielding

$$\hess(f)=\alpha\cdot f^{(n+1)(n-1)}$$
with $\alpha\in\mathbb K^*$. This property was  proved in a similar (but not identical) way   by B. Segre in \cite[Teorema 1]{Segre1}. 
By evaluating
the previous identity on particular matrices B. Segre also deduced $\alpha=(-1)^{\frac{n(n-1)}{2}}n$,
see \cite[Teorema 1]{Segre1}.
\medskip

If $[X]\not\in X_n,$ then $[X^\#]\not\in Y_n$ while for $[X]\in X_n\setminus X_{n-1}$, $[X^\#]\in Y_1$ from which it easily
follows that 
$$X_n^*=Y_1=\p^n\times\p^n\subset \p( (M_{(n+1)\times (n+1)}(\mathbb K))^*).$$

By the definition of the Gauss map, for every $[X]\in X_n\setminus X_{n-1}$, we have that $\nabla_f([X])=[T_{[X]}X_n]$ and,
recalling that the closure of the fibers of the Gauss map are linear spaces, that $\overline{\nabla_f^{-1}([T_{[X]}X_n])}$
is linear space of dimension $n^2-1$. 

We are now ready
to prove the next result.

\begin{prop}\label{conoPnPn} Let notation be as above, let $n\geq 2$ and let $[X]\in X_n\setminus X_{n-1}$. Then
$$\overline{\nabla_f(T_{[X]}X_n)}$$ is a hypersurface of degree $n$ which is a cone with  vertex 
$$\overline{\nabla_f^{-1}([T_{[X]}X_n])}^\perp=\p^{2n}$$ over the dual of a Segre variety  $\p^{n-1}\times\p^{n-1}$.
\end{prop}
\begin{proof} Since $X_n\setminus X_{n-1}$ is  homogeneous, it is sufficient to verify the assertion
for $X$ with $x_{i,j}=\delta_{i,j}$ for $i,j=0,\ldots, n-1$ and $x_{n,j}=x_{i,n}=0$ for every $i,j$. Then
$$X^\#=(0:0:\cdots:0:1) ,$$ so $T_{[X]}X_n$ has equation $x_{n,n}=0$.
Letting $h=\det(Y)$, \eqref{adjform} implies  $\nabla_f^{-1}=\nabla_h$ as rational maps. Then
$$\overline{\nabla_f(V(x_{n,n}))}=V(\frac{\partial h}{\partial y_{n,n}})$$
is the determinant of the $n\times n$ matrix with entries $y_{i,j}$, $i,j=0,\ldots, n-1$ which does not
depend on the $2n+1$ variables $y_{n,i}$ and $y_{j,n}$. Hence it is a cone with vertex $\overline{\nabla_f^{-1}([T_{[X]}X_n])}^\perp$ over the dual of the 
Segre variety $\p^{n-1}\times\p^{n-1}\subset\p^{n^2-1}$ corresponding to the $n^2$ variables $y_{i,j}$,
$i,j=0,\ldots, n-1$. 
\end{proof}
\medskip

\begin{cor}\label{cor4.2} Let notation be as above, let $n\geq 2$, let $$p\in Y_1=\p^n\times \p^n\subset\p^{n^2+2n}=\p((M_{(n+1)\times (n+1)}(\mathbb K))^*)$$ and let $T_n\subset\p^{n^2+2n-1}$ be the projection
of $Y_1=\p^n\times\p^n$ from $p$. Let $q\in X_n\setminus X_{n-1}$ be such that  $T_qX_n=p^\perp$ and let 
$$R_n=X_n\cap p^\perp\subset p^\perp=\p^{n^2+2n-1} .$$ Then:
\medskip
\begin{enumerate}
\item[(i)] $R_n\subset\p^{n^2+2n-1}$ is an irreducible  hypersurface of degree $n+1$ with vanishing hessian and such that $R_n^*=T_n$.
\medskip
\item[(ii)] $Z_{R_n}\subset\p^{n^2+2n-1}$ is a hypersurface of degree $n$ which is a cone with vertex a $\p^{2n-1}$ over  the dual of a Segre variety $\p^{n-1}\times\p^{n-1}$.
\medskip
\item[(iii)] $\codim(R_n^*, Z_{R_n})=n^2-2.$
\end{enumerate}
\end{cor}
\begin{proof} The hypersurface $R_n\subset\p^{n^2+2n-1}$ is connected,  it is also normal by Serre's Criterion being non-singular in codimension 1
($\Sing(R)$ is the closure of the  contact locus $\p^{n^2-1}_q$ of $T_q X_n$
 defined above) and hence it is irreducible. The variety $T_n$ has no dual defect so that $T_n^*$ is an irreducible
hypersurface contained in $R_n$ (see for example \cite[Exercise 1.5.22]{Russo}), yielding $R_n=T_n^*$.   Since $\overline{\nabla_f(T_qX_n)}$ is a cone
such that 
$p\in \vert( \overline{\nabla_f(T_qX_n)})$, we deduce that $Z_{R_n}$ is the projection of $\overline{\nabla_f(T_qX_n)}$ from $p$. Thus (ii) follows
from Proposition \ref{conoPnPn} and Lemma \ref{projectionHess}.  
\end{proof}
\medskip

\begin{rem}\label{CuRS}{\rm The previous result has been discovered  for $n=2$ in \cite{GoRu}, see also \cite[Example 7.6.11]{Russo}.
Part (ii) has been proved algebraically also in \cite[Proposition 4.9]{MoSi}. By passing to a suitable linear section of $X_n$ obtained 
by putting some more variable equal to zero in the matrix $X$ such that $f=\det(X)$, Cunha, Ramos and Simis produced explicit irreducible polynomials with vanishing hessian
and such that  $\codim(X^*,Z_X)$ is a function of $n$. These examples can be also described geometrically as the duals  of some
explicit projections of $Y_1$. 
Clearly the examples in \cite{CuRaSi} are  of Gordan-Noether-Perazzo-Permutti-CRS type since one can  {\it separate the variables}
via Laplace formula for the expansion of the determinant.}
\end{rem}

\subsection{Symmetric determinantal Scorza varieties}
Let
$$W_n=\{S\in M_{(n+1)\times (n+1)}(\mathbb K)\;:\; S=S^t\}\subset M_{(n+1)\times (n+1)}(\mathbb K).$$
 \medskip

Although for $n\geq 1$ the subspace $W_n$ is not a subalgebra of  $M_{(n+1)\times (n+1)}(\mathbb K)$,  we have $S^\#\in W_n$ for every $S\in W_n$.
Let $S=[s_{i,j}]$ be the generic matrix in $W_n$ and let $(s_{0,0}:\cdots:s_{n,n})$ be the corresponding homogeneous coordinates on $$\p(W_n)=
\p^{\frac{n^2+3n}{2}}.$$
Let 
$$g=\det(S)\in\mathbb K[s_{i,j}]_{n+1}.$$
The operation $\#$  on $M_{(n+1)\times (n+1)}(\mathbb K)$ induces by restriction to $W_n$ a birational involution 
$$\nabla_g:\p(W_n)\map\p(W_n^*),$$
defined by $\nabla_g([S])=[S^\#]^t.$

For every $r=1,\ldots, n+1$ we can define the variety 
$$S_r=\{[S]\in\p(W_n)\;:\;\rk(S)\leq r\}\subset \p(W_n);$$
the variety $U_r\subset \p(W_n^*)$ is
defined in the same way. 
With this notation we have $$S_1=\nu_2(\p^n)\subset\p^{\frac{n^2+3n}{2}}=\p(W_n)$$
Veronese  embedded and  that $$S_n=V(\det(S))\subset\p^{\frac{n^2+3n}{2}}$$ is a hypersurface of degree $n+1$. 

The rational map $\nabla_g$ is not  defined along  $S_{n-1}=\Sing(S_n)$ and its  ramification divisor  is given by the determinant of the Jacobian matrix of $\nabla_g$, which is $\hess(g)$.  Since $g=\det(S)$ is an irreducible polynomial,
we deduce from \eqref{adjform} that the ramification divisor of $\nabla_g$ is supported on $S_n$, yielding
$$\hess(g)=\beta\cdot g^{\frac{(n+2)(n-1)}{2}}$$
with $\beta\in\mathbb K^*$. This property was  proved in a similar (but not identical) way   by B. Segre in \cite[Theorem 2]{Segre1}.
By evaluating
the previous identity on particular matrices B. Segre  deduced $\beta=(-1)^{\frac{n(n-1)}{2}}\cdot 2^{\frac{(n+1)n}{2}}\cdot n$,
see \cite[Teorema 2]{Segre1}.
\medskip

If $[S]\not\in S_n,$ then $[S^\#]\not\in U_n$ while for $[S]\in S_n\setminus S_{n-1}$, $[S^\#]\in U_1$. From this it
follows that
$$S_n^*=U_1=\nu_2(\p^n)\subset \p^{\frac{n^2+3n}{2}}=\p(W_n^*).$$
By the definition of the Gauss map, for every $[S]\in S_n\setminus S_{n-1}$, we have $\nabla_g([S])=[T_{[S]}S_n]$ and
$\overline{\nabla_g^{-1}([T_{[S]}S_n])}$
is an $((n^2+n-2)/{2})$-dimensional projective space. 

 We are now ready
to prove the next result and its Corollary, whose proofs will be omitted being analogous to those presented in  Proposition \ref{conoPnPn}
and in Corollary \ref{cor4.2}.

\begin{prop}\label{conov2Pn} Let notation be as above, let $n\geq 2$ and let $[S]\in S_n\setminus S_{n-1}$. Then
$$\overline{\nabla_g(T_{[S]}S_n)}$$ is a hypersuface of degree $n$ which is a cone with  vertex 
$$\overline{\nabla_f^{-1}([T_{[X]}X_n])}^\perp=\p^{n}$$ over the dual of a Veronese variety  $\nu_2(\p^{n-1})$.
\end{prop}
\medskip

\begin{cor}\label{cor4.5} Let notation be as above, let $n\geq 2$, let $$p\in S_1=\nu_2(\p^n)\subset\p^{\frac{n^2+3n}{2}}=\p(W_n^*)$$ and let $V_n\subset\p^{\frac{n^2+3n-2}{2}}$ be the projection
of $S_1=\nu_2(\p^n)$ from $p$. Let $q\in S_n\setminus S_{n-1}$ be such that  $T_qS_n=p^\perp$ and let 
$$Q_n=S_n\cap p^\perp\subset p^\perp=\p^{\frac{n^2+3n-2}{2}} .$$ Then:
\medskip
\begin{enumerate}
\item[(i)] $Q_n\subset\p^{\frac{n^2+3n-2}{2}}$ is an irreducible  hypersurface of degree $n+1$ with vanishing hessian and such that $Q_n^*=V_n$.
\medskip
\item[(ii)] $Z_{Q_n}\subset\p^{\frac{n^2+3n-2}{2}}$ is a hypersurface of degree $n$ which is a cone with vertex a $\p^{n-1}$ over  the dual of a Veronese variety $\nu_2(\p^{n-1})$.
\medskip
\item[(iii)] $\codim(Q_n^*, Z_{Q_n})=\frac{n^2+n-4}{2}.$
\end{enumerate}
\end{cor}
\medskip

\begin{rem}{\rm The varieties $V_n\subset\p^{\frac{n^2+3n-2}{2}}$ are smooth being isomorphic to $\Bl_p\p^n$ and their duals are hypersurfaces
with vanishing hessian. The classification of hypersurfaces with vanishing hessian whose dual is smooth
seems to be an intriguing question, also due to the lack of known examples.

The first element of the series, $V_2\subset \p^4$ is nothing but $S(1,2)$. The cubic hypersurface $S(1,2)^*\subset\p^4$, whose equation we
computed explicitly in Theorem \ref{dualSab}, is surely the easiest and simplest counterexample to Hesse's Claim. As far as we know, the fact that this example was the first member of an infinite series of hypersurfaces with vanishing hessian has been noticed
by us for the first time several years ago. The very recent paper \cite{CuRaSi2} deals with similar phenomena treated from a purely algebraic point of view.}
\end{rem}

\subsection{Skew-symmetric determinantal Scorza varieties}

Let
$$M_n=\{A\in M_{(n+1)\times (n+1)}(\mathbb K)\;:\; A=-A^t\}\subset M_{(n+1)\times (n+1)}(\mathbb K).$$
\medskip

For $n\geq 1$ the subspace $M_n$ is not a subalgebra of  $M_{(n+1)\times (n+1)}(\mathbb K)$ but  $A^\#\in M$ for every $A\in M$.
Let $A=[a_{i,j}]$ be the generic matrix in $M_n$ and let $(a_{0,1}:\cdots:a_{n-1,n})$ be the corresponding homogeneous coordinates on $$\p(M_n)=
\p^{\frac{n^2+n-2}{2}}.$$
From now on suppose that $n+1=2m+2$ with $m\geq 2$ so that $\frac{n^2+n-2}{2}=2m^2+3m$ and $n=2m+1$.
Then 
$$\det(A)=\Pf^2\in\mathbb K[a_{i,j}]_{2m+2},$$ with $\Pf\in \mathbb K[s_{i,j}]_{m+1}$.
The  operation $\#$  on $M_{(n+1)\times (n+1)}(\mathbb K)$ induces by restriction to $M_{2m+1}$ a birational involution 
$$\nabla_{\Pf}:\p(M_{2m+1})\map\p(M_{2m+1}^*),$$
defined by $\nabla_{\Pf}([A])=[A^\#]^t.$

For every $r=1,\ldots, m+1$ we can define the variety 
$$A_{2r}=\{[A]\in\p(M_{2m+1})\;:\;\rk(A)\leq 2r\}\subset \p(M_{2m+1});$$
the variety $C_{2r}\subset \p(M_{2m+1}^*)$ is
defined in the same way. 
With this notation we have $$A_2=\mathbb G(1, \p^{2m+1})\subset\p^{2m^2+3m}=\p(M_{2m+1})$$
Pl\" ucker  embedded and  $$A_{2m}=V(\Pf)\subset\p^{2m^2+3m}$$ is a hypersurface of degree $m+1$. 

The rational map $\nabla_{\Pf}$ is not  defined along  $A_{2m-2}=\Sing(A_{2m})$ and its  ramification divisor  is given by the determinant of the Jacobian matrix of $\nabla_{\Pf}$, which is $\hess(\Pf)$.  Since $\Pf$ is an irreducible polynomial,
we deduce from \eqref{adjform} that the ramification divisor of $\nabla_{\Pf}$ is supported on $A_{2m}$, yielding
$$\hess(\Pf)=\gamma\cdot \Pf^{(2m+1)(m-1)}$$
with $\gamma\in\mathbb K^*$. This property was  proved in a similar (but not identical) way   by B. Segre in \cite[Theorem 3]{Segre1}.
 By evaluating
the previous identity on particular matrices B. Segre  deduced $\gamma=(-1)^{m}\cdot m$,
see \cite[Teorema 3]{Segre1}.
\medskip

If $[A]\not\in A_{2m},$ then $[A^\#]\not\in C_{2m}$ while for $[A]\in A_{2m}\setminus A_{2m-2}$, $[A^\#]\in C_2$ so that
$$A_{2m}^*=C_2=\mathbb G(1,\p^{2m+1})\subset \p^{2m^2+3m}=\p(M_{2m+2}^*).$$

By the definition of the Gauss map, for every $[A]\in A_{2m}\setminus A_{2m-2}$, we have $\nabla_{\Pf}([A])=[T_{[A]}A_{2m}]$ and
$\overline{\nabla_{\Pf}^{-1}([T_{[A]}A_{2m}])}$
is a $(2m^2-m-1)$-dimensional projective space. 

We are now ready
to prove the next result and its Corollary, whose proofs will be omitted being analogous to those presented above.

\begin{prop}\label{conoGn} Let notation be as above, let $m\geq 2$ and let $[A]\in A_{2m}\setminus A_{2m-2}$. Then
$$\overline{\nabla_{\Pf}(T_{[A]}A_{2m})}$$ is a hypersuface of degree $m$ which is a cone with  vertex 
$$\overline{\nabla_f^{-1}([T_{[A]}A_{2m}])}^\perp=\p^{4m}$$ over the dual of a Grassmann variety   $\mathbb G(1,\p^{2m-1})$.
\end{prop}
\medskip 

\begin{cor}\label{cor4.8} Let notation be as above, let $m\geq 2$, let $$p\in C_2=\mathbb G(1, \p^{2m+1})\subset\p^{2m^2+3m}=\p(M_{2m+2}^*)$$ and let $G_m\subset\p^{2m^2+3m-1}$ be the projection
of $C_2$ from $p$. Let $q\in A_{2m}\setminus A_{2m-2}$ be such that  $T_qA_{2m}=p^\perp$ and let 
$$F_m=A_{2m}\cap p^\perp\subset p^\perp=\p^{2m^2+3m-1} .$$ Then:
\medskip
\begin{enumerate}
\item[(i)] $F_m\subset\p^{2m^2+3m-1}$ is an irreducible  hypersurface of degree $m+1$ with vanishing hessian and such that $F_m^*=G_m$.
\medskip
\item[(ii)] $Z_{F_m}\subset\p^{2m^2+3m-1}$ is a hypersurface of degree $m$ which is a cone with vertex a $\p^{4m-1}$ over the dual of  a Grassmann variety $\mathbb G(1,\p^{2m-1})$.
\medskip
\item[(iii)] $\codim(F_m^*, Z_{F_m})=2m^2-m-2.$
\end{enumerate}
\end{cor}
\medskip

\begin{rem}{\rm The three series of hypersurfaces $R_n$, $Q_n$ and $F_m$ are of Gordan-Noether-Perazzo-Permutti-CRS type and such that the duals of their polar images are of
the same type of their duals. Indeed, the first property easily  follows from Laplace formula and by the determinantal description of their equation in a suitable coordinate system
while $Z^*$ is of the same type by part (ii) of the previous Corollaries~\ref{cor4.2}, \ref{cor4.5}, and \ref{cor4.8}. }
\end{rem}


\begin{thebibliography}{CuRaSi2}

\bibitem[AG1]{AG} M. A. Akivis, V. V. Goldberg, {\it Differential Geometry of Varieties with Degenerate Gauss Maps}, CMS books in Mathematics, 2004.
\bibitem[AG2]{AG2} M. A. Akivis, V. V. Goldberg, {\it Smooth Lines on Projective Planes over
Two-Dimensional Algebras and Submanifolds with Degenerate Gauss Maps } Beitr\" age zur Algebra und Geometrie
"Contributions to Algebra and Geometry",  Volume 44 (2003), No. 1, 165-178.

\bibitem[CiRuSi]{CRS} C. Ciliberto, F. Russo, A. Simis, {\it Homaloidal hypersurfaces and hypersurfaces with vanishing hessian}, Adv.\ in\ Math.
 218 (2008), 1759--1805.
 
\bibitem[CLO]{Cox} D. A. Cox, J. Little, and D. O'Shea, {\it Using algebraic geometry}, 2nd ed., Graduate Texts in Mathematics 185, Springer, 2005.


\bibitem[CuRaSi1]{CuRaSi} R. Cunha, Z. Ramos, A. Simis, {\it Degenerations of the generic square matrix, the polar map and determinantal structure}, Internat. J. Algebra Comput. 28 (2018), 1255--1297.

\bibitem[CuRaSi2]{CuRaSi2} R. Cunha, Z. Ramos, A. Simis, {\it Symmetry preserving degenerations of the generic symmetric matrix},
J. Algebra 523  (2019), 154--191.


\bibitem[DBo1]{DeBondt1} M. De Bondt, {\it Quasi-translations and  counterexamples to the homogeneous dependence problem},  Proc. Amer. Math. Soc. 134 (2006), no. 10, 2849--2856

\bibitem[DBo2]{DeBondt2} M. De Bondt, {\it Quasi-translations and singular hessians}, Colloq. Math.152 (2018), 175--198.


\bibitem[Fra]{Franchetta} A. Franchetta, {\it Sulle forme algebriche di $S_4$ aventi l'hessiana indeterminata}, Rend. Mat. 13 (1954), 1--6.

\bibitem[FP]{FP} G.\ Fischer, J.\ Piontkowski, {\it Ruled Varieties -- An Introduction to Algebraic
Differential Geometry}, Advanced Lectures in Mathematics, Vieweg, 2001.



\bibitem[GaRe]{GR} A.\ Garbagnati, F.\ Repetto, {\it A geometrical approach to
Gordan--Noether's and Franchetta's contributions to a question posed by Hesse}, Collect.\ Math.,
60 (2009), 27--41.

\bibitem[Gon]{Rodrigo} R. Gondim, {\it On higher hessians and the Lefschetz properties}, J. Algebra 489 (2017), 241--263.


\bibitem[GoRu]{GoRu} R. Gondim, F. Russo, {\it On cubic hypersurfaces with vanishing hessian}, Jour. Pure and Applied Algebra 219 (2015), 779--806.

\bibitem[GoZa]{GoZa} R. Gondim, G. Zappal\` a, {\it Lefschetz properties for Artinian Gorenstein algebras presented by quadrics}, Proc. A.M.S., vol. 146, n. 3 (2018), 993--1003.


\bibitem[GoNo]{GN} P. Gordan, M. Noether, {\it \"Ueber die algebraischen Formen, deren Hesse'sche Determinante identisch verschwindet}, Math. Ann. 10 (1876), 547--568.

\bibitem[GKZ]{GKZ}  I. M. Gelfand, M. M. Kapranov, A. V. Zelevinsky, {\it Discriminants, Resultants and Multidimensional Determinants}, Springer New York 1994.


\bibitem[He1]{Hesse1} O. Hesse, {\it \"Uber die Bedingung, unter  welche eine homogene ganze Function von nunabh\"angigen Variabeln durch Line\"are Substitutionem von n andern unabh\"angigen Variabeln auf eine homogene Function sich zur\"uck-f\"uhren l\"asst, die eine Variable weniger enth\"alt},  J. reine angew. Math. 42 (1851), 117--124.


\bibitem[He2]{Hesse2} O. Hesse, {\it Zur Theorie der ganzen homogenen Functionem},  J. reine angew. Math. 56 (1859), 263--269.

\bibitem[Koh]{Kohn} K. Kohn, {\it Coisotropic Hypersurfaces in the Grassmannian}, preprint, http://arxiv.org/abs/1607.05932.

\bibitem[Los]{Lossen} C. Lossen, {\it When does the hessian determinant vanish identically? (On Gordan and Noether's Proof of Hesse's Claim)}, Bull. Braz.
Math. Soc. 35 (2004), 71--82.

\bibitem[MoSi]{MoSi} M. Mostafazadehfard, A. Simis, {\it Homaloidal determinants}, J. Algebra 450 (2016), 59--101.


\bibitem[Per]{Perazzo} U. Perazzo, {\it Sulle variet\'a cubiche la cui hessiana svanisce
identicamente}, Giornale di Matematiche (Battaglini) 38
(1900), 337--354.
\bibitem[Pm1]{Permutti1} R. Permutti, {\it Su certe forme a hessiana indeterminata}, Ricerche di Mat. 6 (1957), 3--10.

\bibitem[Pm2]{Permutti2} R. Permutti, {\it Su certe classi di forme a hessiana indeterminata}, Ricerche di Mat. 13 (1964), 97--105.

\bibitem[Rus]{Russo}  F. Russo, {\it On the Geometry of Some Special Projective Varieties}, Lecture Notes of the Unione Matematica Italiana, Springer Verlag, Berlin, 2016.

\bibitem[Se1]{Segre} B.\ Segre, {\it Bertini forms and hessian
matrices}, J.\ London Math.\ Soc.\ 26 (1951), 164--176.


\bibitem[Se2]{Segre1} B.\ Segre, {\it Sull' hessiano di taluni
polinomi (determinanti, pfaffiani, discriminanti, risultanti,
hessiani) I, II},  Atti Acc.\ Lincei 37 (1964), 109--117 and
215--221.



\bibitem[WZ]{WZ} J. Weyman, A. Zelevinsky, {\it Multiplicative properties of projectively dual varieties}, Man. Math. 82 (1994), 139--148.


\end{thebibliography}
\end{document}